\documentclass[letterpaper,10pt,reqno,onefignum,onetabnum]{amsart}
\usepackage[english]{babel}
\usepackage{amsmath}
\usepackage{amsthm}
\usepackage{verbatim}
\usepackage{mathrsfs}
\usepackage{bm}
\usepackage[dvipsnames]{xcolor}
\usepackage{hyperref}
\hypersetup{
     colorlinks = true,
     linkcolor = OliveGreen,
     anchorcolor = OliveGreen,
     citecolor = OliveGreen,
     filecolor = OliveGreen,
     urlcolor = OliveGreen
     }
\usepackage{algorithm}
\usepackage{algorithmic}
\usepackage{float}
\usepackage{lipsum}
\usepackage{amsfonts}
\usepackage{amssymb}
\usepackage{graphicx}
\usepackage{epstopdf}
\usepackage{cases}
\usepackage{multirow}
\ifpdf
  \DeclareGraphicsExtensions{.eps,.pdf,.png,.jpg}
\else
  \DeclareGraphicsExtensions{.eps}
\fi
\usepackage{verbatim}
\usepackage{mathrsfs}
\usepackage{bm}

\newtheorem{teo}{Theorem}[section]
\newtheorem{prop}[teo]{Proposition}

\newtheorem{cor}[teo]{Corollary}
\newtheorem{pro}[teo]{Problem}
\newtheorem{algo}[teo]{Algorithm}

\newtheorem{rem}[teo]{Remark}
\newtheorem{ejem}[teo]{Example}

\usepackage{lipsum}
\usepackage{amsfonts}
\usepackage{amssymb}
\usepackage{graphicx}
\usepackage{epstopdf}

\usepackage{cases}
\usepackage{multirow}
\usepackage{algorithmic}
\usepackage{tikz}
\usepackage{booktabs}%
\usetikzlibrary{matrix}
\usetikzlibrary{arrows}
\ifpdf
  \DeclareGraphicsExtensions{.eps,.pdf,.png,.jpg}
\else
  \DeclareGraphicsExtensions{.eps}
\fi
\usepackage{verbatim}
\usepackage{mathrsfs}
\usepackage{bm}
\usepackage{color}
\usepackage{caption}
\usepackage{subcaption}
\usepackage{enumitem}

\newcommand{\N}{\mathbb N}

\newcommand{\R}{\mathbb R}

\renewcommand{\H}{\mathcal{H}}
\newcommand{\G}{\mathcal G}

\newcommand{\sH}{\mathsf H}
\newcommand{\sG}{\mathsf G}
\newcommand{\sx}{\mathsf x}

\newcommand{\sr}{\mathsf r}
\newcommand{\sg}{\mathsf g}

\newcommand{\sy}{\mathsf y}

\newcommand{\sq}{\mathsf q}

\newcommand{\su}{\mathsf u}
\newcommand{\sA}{\mathsf A}
\newcommand{\sB}{\mathsf B}
\newcommand{\sC}{\mathsf C}
\newcommand{\sD}{\mathsf D}
\newcommand{\sT}{\mathsf T}
\newcommand{\sM}{\mathsf M}
\newcommand{\sN}{\mathsf N}
\newcommand{\sL}{\mathsf L}
\newcommand{\sV}{\mathsf V}

\newcommand{\id}{\textnormal{Id}}

\newcommand{\weak}{\rightharpoonup}
\newcommand{\ran}{\textnormal{ran}\,}
\newcommand{\dom}{\textnormal{dom}\,}
\newcommand{\infconv}{\ensuremath{\mbox{\small$\,\square\,$}}}

\newcommand{\zer}{\textnormal{zer}}

\newcommand{\gra}{\textnormal{gra}\,}

\newcommand{\argm}[1]{\underset{#1}{\argmin\, }}

\newcommand{\scal}[2]{{\left\langle{{#1}\mid{#2}}\right\rangle}}

\newcommand{\pnorm}[1]{|\hspace{-.3mm}|\hspace{-.3mm}|{#1}|\hspace{-.3mm}|\hspace{-.3mm}|}

\newcommand{\menge}[2]{\big\{{#1}~\big |~{#2}\big\}}
\newcommand{\pinf}{\ensuremath{{+\infty}}}

\newcommand{\RR}{\ensuremath{\mathbb{R}}}

\newcommand{\RPP}{\ensuremath{\left]0,+\infty\right[}}

\newcommand{\RX}{\ensuremath{\left]-\infty,+\infty\right]}}

\newcommand{\sri}{\ensuremath{\text{\rm sri}\,}}

\newcommand{\prox}{\ensuremath{\text{\rm prox}\,}}

\newcommand{\weakly}{\ensuremath{\:\rightharpoonup\:}}
\usepackage{geometry}
\numberwithin{equation}{section}

\DeclareFontEncoding{FMS}{}{}
\DeclareFontSubstitution{FMS}{futm}{m}{n}
\DeclareFontEncoding{FMX}{}{}
\DeclareFontSubstitution{FMX}{futm}{m}{n}
\DeclareSymbolFont{fouriersymbols}{FMS}{futm}{m}{n}
\DeclareSymbolFont{fourierlargesymbols}{FMX}{futm}{m}{n}
\DeclareMathDelimiter{\nr}{\mathord}{fouriersymbols}{152}{fourierlargesymbols}{147}

\DeclareMathOperator*{\argmin}{arg\,min}
\DeclareMathDelimiter{\nr}{\mathord}{fouriersymbols}{152}{fourierlargesymbols}{147}
\DeclareMathAlphabet{\mathpzc}{OT1}{pzc}{m}{it}

\title[Forward-partial inverse-half-forward splitting
algorithm]{Forward-partial inverse-half-forward splitting algorithm for
solving monotone inclusions}

\author{Luis M. Brice\~no-Arias$^{1}$, Jinjian Chen$^{2}$, Fernando Rold\'an$^{1}$, and Yuchao Tang$^{2}$}
\address{1. Departamento de Matem\'{a}tica, Universidad T\'{e}cnica Federico Santa Mar\'{i}a, Avenida Espa\~{n}a 1680, Valpara\'{i}so, Chile
 \\
2. Department of Mathematics, Nanchang University, Nanchang 330031, P.R. China}
\email{luis.briceno@usm.cl, fernando.roldan@usm.cl}

\subjclass[2010]{47H05, 47H10, 65K05, 65K15, 90C25, 49M29.}

\begin{document}

\begin{abstract}
In this paper we provide a splitting
algorithm for solving coupled monotone inclusions in a real Hilbert
space involving the sum of a normal cone to a vector subspace, a
maximally monotone, a monotone-Lipschitzian, and a cocoercive
operator. The proposed method takes advantage of the intrinsic
properties of each operator and generalizes the method of partial
inverses and the forward-backward-half forward splitting, among other
methods. At each iteration, our algorithm needs two computations of
the Lipschitzian operator while the cocoercive operator is activated
only once. By using product space techniques, we derive a
method for solving a composite monotone primal-dual inclusions
including linear operators and we apply it to solve constrained
composite convex optimization problems. Finally, we
apply our algorithm to a constrained total variation least-squares
problem and we compare its performance with
efficient methods in the literature.
\par
\bigskip

\noindent \textbf{Keywords.} {\it Splitting algorithms, monotone
operator theory, partial inverse, convex optimization.}
\end{abstract}

\maketitle
\section{Introduction}

In this paper we study the numerical resolution of the following
inclusion problem. The normal cone to $V$ is denoted by $N_V$.	
\begin{pro}\label{prob:main}
	Let $\H$ be a real Hilbert space and let $V$ be a closed vector
	subspace of $\H$.  Let $A\colon \H \to 2^{\H}$ be a maximally
	monotone operator, let $B\colon \H \to \H$ be a monotone and
	$L$-Lipschitzian operator for some $L\in\RPP$, and let $C \colon
	\H \to \H$ be a $\beta$-cocoercive operator for some
	$\beta\in\RPP$.
	The problem is to
	\begin{equation}\label{eq:primalinclusionnv}
		\text{find }\quad x \in \H\quad \text{ such that } \quad 0 \in Ax+
		Bx
		+Cx+N_V x,
	\end{equation}
	under the assumption that its solutions set $Z$ is nonempty.
\end{pro}
Problem~\ref{prob:main} models a wide class of problems in
engineering
including mechanical problems
\cite{Gabay1983,Glowinski1975,Goldstein1964}, differential
inclusions \cite{AubinHelene2009,Showalter1997}, game theory
\cite{AttouchCabot19,Nash13},
restoration and denoising in image processing
\cite{chambolle1997,Pustelnik2019,daubechies2004}, traffic theory
\cite{Nets1,Fukushima1996The,GafniBert84},
among
others.

In the case when $V=\H$ and the resolvent of $B$ is available,
Problem~\ref{prob:main}
can be solved by the algorithms in \cite{davis2015,Dong2021} and, if
$B$ is linear, by the algorithm in \cite{LatafatPatrinos2017}.
Moreover, if the resolvent of $B$ is difficult to compute,
Problem~\ref{prob:main} can be solved by the
\textit{forward-backward-half forward} algorithm (FBHF) proposed in
\cite{Arias2017A}. FBHF implement explicit activations of $B$ and $C$
and generalizes the classical forward-backward splitting
\cite{Lions1979SIAM} and Tseng's splitting \cite{Tseng2000SIAM}
when $B=0$ and $C=0$, respectively.

In the case when $V\neq \H$, a splitting algorithm for solving the case
$B=C=0$ is proposed in \cite{Spingran1983AMO} using the partial
inverse of
$A$ with respect to $V$ and extensions for the
cases $B=0$ and $C=0$ are proposed in \cite{briceno2015Optim}
and \cite{Briceno2015JOTA}, respectively.
On the other hand, the algorithms proposed in
\cite{Bot-2019-ACM,BotHendrich2013,Bot2016NA,Bot2013,Bot2015AMC,briceno2011SIAM,Cevher2020SVVA,CombPes12,Combettes2018MP,Combettes2014Optimization,Csetnek2019AMO,davis2015,DungVu15,Eckstein2017,Eckstein2020,EcksteinJohnstone,Malitsky2020SIAMJO,Raguet-SIAM-2013,Rieger2020AMC,RyuVu2020,Vu13}
can solve Problem~\ref{prob:main} under additional assumptions
or without exploiting the
vector subspace structure and the intrinsic properties of the operators
involved.
Indeed, the algorithms in
\cite{BotHendrich2013,Bot2013,Bot2015AMC,briceno2011SIAM,Combettes2018MP,Eckstein2017}
need to compute the resolvents of $B$ and $C$, which are not explicit
in general or they can be
numerically expensive. In addition, previous methods do not take
advantage of the vector subspace structure of
Problem~\ref{prob:main}. The schemes proposed in
\cite{Bot2016NA,CombPes12,DungVu15,Eckstein2020} take
advantage of the
properties of $B$, but the cocoercivity of $C$ and the vector
subspace structure are not leveraged. In fact, the algorithms in
\cite{Bot2016NA,CombPes12,DungVu15,Eckstein2020} may
consider
$B+C$ as a monotone and Lipschitzian operator and activate it twice
by iteration.
In contrast, the algorithms in
\cite{Cevher2020SVVA,Csetnek2019AMO,Malitsky2020SIAMJO,Rieger2020AMC,RyuVu2020}
activates $B+C$ only once by iteration, but they need to store in the
memory the two past iterations and the step-size is reduced
significantly.
In addition,
the methods proposed in
\cite{Bot-2019-ACM,Combettes2014Optimization,davis2015,EcksteinJohnstone,Raguet-SIAM-2013,Vu13}
take advantage of the cocoercivity of $C$, but they do not exploit
neither the
properties of $B$ nor the vector subspace structure of the problem.

Furthermore, note that Problem~\ref{prob:main} can be solved by the
algorithms proposed in \cite{Arias2017A,CombettesMinh2022}
by considering $N_V$ as any maximally monotone operator via
product space techniques. These
approaches do not exploit the
vector subspace structure of the problem and need to update
additional auxiliary dual variables at each iteration, which affects
their efficiency in large scale problems.
Moreover, since $B+C$ is monotone and
$(\beta^{-1}+L)$-Lipschitzian, Problem~\ref{prob:main} can be solved
by \cite{Briceno2015JOTA}. However, this implementation needs two
computations of $C$ by iteration which affects its efficiency when
$C$ is computationally expensive and also may increment drastically
the number of iterations to achieve the convergence criterion, as
perceived in \cite[Section~7.1]{Arias2017A} in the case $V=\H$.

In this paper we propose a splitting algorithm which
fully exploits the vector subspace structure, the cocoercivity of $C$,
and the Lipschitzian property of $B$. In the particular case when
$V=\H$, we recover \cite{Arias2017A}, which generalizes the
forward-backward splitting and Tseng's splitting
\cite{Tseng2000SIAM}. For general vector subspaces, our algorithm
also recovers the methods proposed in
\cite{briceno2015Optim,Briceno2015JOTA,Spingran1983AMO}.
By using standard product space techniques, we apply our algorithm
to solve composite primal-dual monotone inclusions including a
normal cone to a vector subspace, cocoercive, and
Lipschitzian-monotone operators and composite convex
optimization problems under vector subspace constraints.
We implement our method in the context of TV-regularized
least-squares problems with constraints and we compare its
performance with previous methods in the literature including
\cite{Condat13}. We observe that, in the case when the matrix in the
data fidelity term has large norm values, our implementation is more
efficient.

The paper is organized as follows. In Section~\ref{sec:notation} we
set our notation. In Section~\ref{sec:main} we provide our main
algorithm for solving Problem~\ref{prob:main} and its proof of
convergence. In Section~\ref{sec:application} we derive a method for
solving a composite monotone primal-dual inclusion, including
monotone, Lipschitzian, cocoercive, and bounded linear operators. In
this section we also derive an algorithm for solve constrained
composite convex optimization problems. Finally, in
Section~\ref{sec:numerical} we provide numerical experiments
illustrating the efficiency of our proposed method.

\section{Notations and Preliminaries} \label{sec:notation}
Throughout this paper $\H$ and $\G$ are real Hilbert spaces. We
denote their scalar
products by $\scal{\cdot}{\cdot}$, the associated norms by
$\|\cdot \|$, and by $\weakly$ the weak convergence.
Given a linear bounded
operator $L:\H \to \G$, we denote its adjoint by $L^*\colon\G\to\H$. $\id$ denotes the identity operator on $\H$.
Let $D\subset \H$ be non-empty and let $T: D \rightarrow \H$.
Let $\beta \in \left]0,+\infty\right[$. The operator $T$ is $\beta-$cocoercive if
\begin{equation} \label{def:coco}
	(\forall x \in D) (\forall y \in D)\quad \langle x-y \mid Tx-Ty \rangle
	\geq \beta \|Tx - Ty
	\|^2
\end{equation}
and it is $L-$Lipschitzian if
\begin{equation} \label{def:lips}
	(\forall x \in D) (\forall y \in D)\quad \|Tx-Ty\| \leq  L\|x - y \|.
\end{equation}	
Let $A:\H \rightarrow 2^{\H}$ be a set-valued operator. The domain,
range, and graph of $A$ are
$\dom\, A = \menge{x \in \H}{Ax \neq  \varnothing}$,
$\ran\, A = \menge{u \in \H}{(\exists x \in \H)\, u \in Ax}$,
and $\gra
A = \menge{(x,u) \in \H \times \H}{u \in Ax}$,
respectively.
The set of zeros of $A$ is  $\zer A =
\menge{x \in \H}{0 \in Ax}$, the inverse of $A$ is $A^{-1}\colon \H  \to 2^\H \colon u
\mapsto
\menge{x \in \H}{u \in Ax}$, and the resolvent of $A$ is
$J_A=(\id+A)^{-1}$.
The operator $A$  is monotone if
\begin{equation}\label{def:monotone}
	(\forall (x,u) \in \gra A) (\forall (y,v) \in \gra A)\quad \scal{x-y}{u-v}
	\geq 0
\end{equation}
and it is maximally monotone if it is monotone and there exists no
monotone operator $B :\H\to  2^{\H}$ such that $\gra B$ properly
contains $\gra A$, i.e., for every $(x,u) \in \H \times \H$,
\begin{equation} \label{def:maxmonotone}
	(x,u) \in \gra A \quad \Leftrightarrow\quad  (\forall (y,v) \in \gra A)\
	\
	\langle x-y \mid u-v \rangle \geq 0.
\end{equation}
We denote by $\Gamma_0(\H)$ the class of proper lower
semicontinuous convex functions $f\colon\H\to\RX$. Let
$f\in\Gamma_0(\H)$.
The Fenchel conjugate of $f$ is
defined by $f^*\colon u\mapsto \sup_{x\in\H}(\scal{x}{u}-f(x))$, which
is a function in $\Gamma_0(\H)$,
the subdifferential of $f$ is the maximally monotone operator
$$\partial f\colon x\mapsto \menge{u\in\H}{(\forall y\in\H)\:\:
	f(x)+\scal{y-x}{u}\le f(y)},$$
we have that $(\partial f)^{-1}=\partial f^*$,
and that $\zer\,\partial f$ is the set of
minimizers of $f$, which is denoted by $\arg\min_{x\in \H}f$.
We denote by
\begin{equation}
	\label{e:prox}
	\prox_{f}\colon
	x\mapsto\argm{y\in\H}\Big(f(y)+\frac{1}{2}\|x-y\|^2\Big).
\end{equation}
We have
$\prox_f=J_{\partial f}$. Moreover, it follows from \cite[Theorem 14.3]{bauschkebook2017} that
\begin{equation}
	\label{e:Moreau_nonsme}
	\prox_{\gamma f}+\gamma \prox_{f^*/\gamma} \circ
	\id/\gamma=\id.
\end{equation}
Given a non-empty closed convex set $C\subset\H$, we denote by
$P_C$ the projection onto $C$, by
$\iota_C\in\Gamma_0(\H)$ the indicator function of $C$, which
takes the value $0$ in $C$ and $\pinf$ otherwise, and by $N_C=
\partial (\iota_C)$ the normal cone to $C$. The partial inverse of $A$
with respect to a closed vector subspace $V$ of $\H$, denoted by
$A_V$, is
defined by
\begin{equation}\label{eq:defparinv}
	(\forall (x,y) \in \H^2) \quad y \in A_V x \quad
	\Leftrightarrow\quad
	(P_V y + P_{V^{\bot}}x) \in A(P_V x +P_{V^{\bot}}y).
\end{equation}
Note that $A_\H = A$ and $A_{\{0\}} = A^{-1}$.
For further properties of monotone operators,
non-expansive mappings, and convex analysis, the
reader is referred to \cite{bauschkebook2017}.

The following is a simplified version of the algorithm proposed in
\cite[Theorem~2.3]{Arias2017A}.
\begin{prop}{\cite[Theorem~2.3]{Arias2017A}}\label{prop:pre}
	Let $\hat{L}\in\RPP$, let $\hat{\beta}\in\RPP$, let
	$\mathcal{A}\colon
	\H \to 2^\H$
	be a maximally monotone
	operator, let $\mathcal{B}\colon \H \to \H$ be monotone and
	$\hat{L}$-Lipschitzian, and let $\mathcal{C}\colon \H \to \H$ be a
	$\hat{\beta}$-cocoercive operator. Suppose that
	$\zer(\mathcal{A}+\mathcal{B}+\mathcal{C})\ne\varnothing$ and
	set
	\begin{equation}\label{e:chipre}
		\hat{\chi} = \frac{4\hat{\beta}}{1+\sqrt{1+16\hat{\beta}^2
				\hat{L}^2}}
		\in \left] 0,
		\min\left\{2\hat{\beta}, \frac{1}{\hat{L}}\right\}\right[,
	\end{equation}
	let $(\lambda_{n})_{n \in \N}$ be a sequence in
	$[\varepsilon,\hat{\chi}-\varepsilon]$, for some $\varepsilon \in
	\left]0,\hat{\chi}/2\right[$. Moreover, let $z_0 \in \H$ and consider
	the
	following recurrence
	\begin{equation}
		\label{e:algopre}
		\left\lfloor
		\begin{aligned}
			&\text{for }n=1,2,\ldots\\
			&s_n = J_{\lambda_n \mathcal{A}} \big(z_n-\lambda_n
			(\mathcal{B} +\mathcal{C}) z_n\big)\\
			&z_{n+1} = s_n + \lambda_n (\mathcal{B} z_n -
			\mathcal{B} s_n).
		\end{aligned}
		\right.
	\end{equation}
	Then, $(z_n)_{n\in\N}$ converges weakly to some $\overline{z}
	\in
	\zer
	(\mathcal{A}+\mathcal{B}+\mathcal{C})$.
\end{prop}
Observe that \eqref{e:algopre} reduces to forward-backward splitting
when $\mathcal{B}=0$ (and $L=0$), and to a version of Tseng's
splitting when
$\mathcal{C}=0$ (and $\beta\to\pinf$)
\cite{briceno2011SIAM,Tseng2000SIAM}.
\section{Main result}\label{sec:main}
The following is our main algorithm, whose convergence is
proved in Theorem~\ref{teo:main1} below.

\begin{algo}\label{algo:1} In the context of Problem~\ref{prob:main},
	let $(x_{0},y_{0})\in V\times V^{\bot}$, let $\gamma\in\RPP$, and
	let
	$(\lambda_{n})_{n\in\N}$ be a sequence in $\RPP$. Consider
	the recurrence
	\begin{equation}
		\label{e:algomain}
		\left\lfloor
		\begin{aligned}
			&\text{for }n=1,2,\ldots\\
			&\text{ find }\, (p_{n},q_{n})\in \mathcal{H}^2\, \text{ such that }\,
			x_{n}+\gamma y_{n}-\lambda_{n}\gamma
			P_{V}(B+C)x_{n}=p_{n}+\gamma
			q_{n}\\
			& \quad \quad \text{ and } \,
			\frac{P_{V}q_{n}}{\lambda_{n}}+P_{V^{\bot}}q_{n}\in
			A\left(P_{V}p_{n}+\frac{P_{V^{\bot}}p_{n}}{\lambda_{n}}\right),\\
			& x_{n+1}=P_{V}p_{n}+\lambda_{n}\gamma
			P_{V}(Bx_{n}-BP_{V}p_{n}),\\
			&y_{n+1}=P_{V^{\bot}}q_{n}.
		\end{aligned}
		\right.
	\end{equation}
\end{algo}

Note that \eqref{algo:1} involves only one activation of $C$, two of
$B$,
and three projections onto $V$ at each iteration.

\begin{teo}\label{teo:main1}
	In the context of Problem~\ref{prob:main}, set
	\begin{equation}
		\label{e:defchi}
		\chi=\frac{4\beta}{1+\sqrt{1+16\beta^2
				L^2}}\in\left]0,\min\left\{2\beta,\frac{1}{L}\right\}\right[,
	\end{equation}
	let $\gamma\in\RPP$,
	and let $(\lambda_{n})_{n\in\N}$ be a sequence in
	$[\varepsilon,\chi/\gamma-\varepsilon]$ for some
	$\varepsilon\in\left]0,\chi/(2\gamma)\right[$.
	Moreover, let $(x_{0},y_{0})\in V\times V^{\bot}$ and let $(x_n)_{n
		\in
		\N}$ and
	$(y_n)_{n \in \N}$ be the sequences generated by
	Algorithm~\ref{algo:1}. Then $(x_{n})_{n\in\N}$ and
	$(y_n)_{n \in \N}$ are sequences in $V$ and
	$V^{\bot}$, respectively, and there exist
	$\overline{x}\in Z$ and $\overline{y}\in
	V^{\bot}\cap(A\overline{x}+P_{V}(B+C)\overline{x})$
	such that $x_{n}\weak \overline{x}$ and $y_{n}\weak
	\overline{y}$.
\end{teo}
\begin{proof}
	Define 	
	\begin{equation}\label{eq:defABC}
		\begin{cases}
			\mathcal{A}_\gamma = (\gamma A)_V \colon \H \to 2^\H\\
			\mathcal{B}_\gamma = \gamma P_V \circ B \circ P_V \colon \H
			\to \H\\
			\mathcal{C}_\gamma = \gamma P_V \circ C \circ P_V \colon \H
			\to \H.
		\end{cases}
	\end{equation}
	It follows from \cite[Proposition~3.1(i)\&(ii)]{Briceno2015JOTA}
	that
	$\mathcal{A}_\gamma$ is maximally monotone and that
	$\mathcal{B}_\gamma$ is monotone and $\gamma
	L$-Lipschitzian.
	Moreover, $\mathcal{C}_\gamma$
	is $\beta/\gamma$-cocoercive in view of
	\cite[Proposition~5.1(ii)]{briceno2015Optim}.
	Since $C$ is $\beta^{-1}$-Lipschitzian, $B+C$ is $
	(\beta^{-1}+L)$-Lipschitzian, and \eqref{eq:defABC} and the
	linearity
	of $P_V$ yield
	\begin{equation}
		\label{e:ident}
		\mathcal{B}_{\gamma}+\mathcal{C}_{\gamma}=\gamma P_V\circ
		(B+C)\circ P_V.
	\end{equation}
	Therefore,
	\cite[Proposition~3.1(iii)]{Briceno2015JOTA} implies that
	$\hat{x} \in \H$ is a
	solution to Problem~\ref{prob:main} if and only if
	\begin{multline} \label{eq:solABC}
		\hat{x} \in V\quad \text{and}\quad
		\big(\exists \hat{y} \in V^\bot\cap
		(A\hat{x}+B\hat{x}+C\hat{x})\big)\\
		\quad \hat{x} + 	\gamma\big(\hat{y}-P_{V^\bot} (B+C)
		\hat{x}\big) \in
		\zer
		(\mathcal{A}_\gamma+\mathcal{B}_{\gamma}+\mathcal{C}_\gamma).
	\end{multline}
	Now, since $x_0 \in V$ and $y_0 \in V^{\bot}$, it follows from
	Algorithm~\ref{e:algomain}  that $(x_n)_{n \in \N}$ and
	$(y_n)_{n \in \N}$ are sequences in $V$ and $V^{\bot}$,
	respectively. In addition, from
	Algorithm~\ref{e:algomain} and
	\cite[Proposition~3.1(i)]{Briceno2015JOTA} we deduce that
	\begin{equation}\label{eq:rsvAg}
		(\forall n\in\N)\quad 	J_{\lambda_n \mathcal{A}_\gamma} (x_n +
		\gamma y_n -\lambda_n\gamma P_V( B + C)x_n) = P_V p_n +
		\gamma P_{V^\bot} q_n.	
	\end{equation}
	For every $n \in \N$, set $z_n = x_n +\gamma y_n$ and set
	$s_n=P_V p_n + \gamma P_{V^\bot} q_n$. Hence, for every
	$n\in\N$,
	$P_V s_n = P_V p_n$, $P_{V^{\bot}} s_n= \gamma
	P_{V^{\bot}} q_n$, and
	\eqref{eq:rsvAg} and \eqref{e:ident} yield
	\begin{align}\label{eq:defsn}
		s_n&=J_{\lambda_n \mathcal{A}_\gamma} (x_n + \gamma y_n
		-\lambda_n\gamma P_V( B +  C)x_n)\nonumber\\
		&=J_{\lambda_n \mathcal{A}_\gamma} (z_n -\lambda_n\gamma
		P_V( B +  C)P_Vz_n)\nonumber\\
		&=J_{\lambda_n \mathcal{A}_\gamma} \big(z_n-\lambda_n
		(\mathcal{B}_\gamma +\mathcal{C}_\gamma) z_n\big).
	\end{align}
	Thus, from Algorithm~\ref{e:algomain} we deduce that, for every
	$n\in\N$,
	\begin{align}\label{eq:zn+1}
		z_{n+1}&=x_{n+1}+\gamma y_{n+1}\nonumber\\
		&= P_V p_n + \lambda_n \gamma P_V(Bx_n - B P_V p_n) +
		\gamma
		P_{V^{\bot}} q_n\nonumber\\
		&= P_{V} s_n +  \lambda_n (\gamma P_VBP_Vz_n -
		\gamma P_V
		B  P_V s_n)+ P_{V^{\bot}} s_n\nonumber\\
		&= s_n + \lambda_n (\mathcal{B}_{\gamma} z_n -
		\mathcal{B}_{\gamma}  s_n).
	\end{align}
	Therefore, we obtain from  \eqref{eq:defsn} and \eqref{eq:zn+1}
	that
	\begin{equation}
		\left\lfloor
		\begin{aligned}
			&\text{for }n=1,2,\ldots\\
			&s_n = J_{\lambda_n \mathcal{A}_\gamma}
			\big(z_n-\lambda_n
			(\mathcal{B}_\gamma +\mathcal{C}_\gamma) z_n\big)\\
			&z_{n+1} = s_n + \lambda_n (\mathcal{B}_\gamma z_n
			-\mathcal{B}_\gamma s_n).
		\end{aligned}
		\right.
	\end{equation}
	Altogether, by setting $\hat{\beta}=\beta/\gamma$ and
	$\hat{L}=\gamma L$, we have $\hat{\chi}=\chi/\gamma$ and
	Proposition~\ref{prop:pre} asserts that there exists $\overline{z}
	\in
	\zer (\mathcal{A}_\gamma + \mathcal{B}_\gamma
	+\mathcal{C}_\gamma)$ such that
	$z_n \weak
	\overline{z}$. Furthermore, by setting $\overline{x} = P_V
	\overline{z}$ and $\overline{y} = P_{V^\bot}
	\overline{z}/\gamma$,
	we have $-(\mathcal{B}_\gamma
	+\mathcal{C}_\gamma)(\overline{x}+\gamma\overline{y})\in
	\mathcal{A}_\gamma(\overline{x}+\gamma\overline{y})$,
	which, in view of \eqref{eq:defABC}, is equivalent to
	$-P_V(B+C)\overline{x}+\overline{y}\in A\overline{x}$. Therefore, by
	defining $\hat{y}=\overline{y}+P_{V^{\bot}}(B+C)\overline{x}\in
	V^{\bot}\cap (A\overline{x}+B\overline{x}+C\overline{x})$, we have
	$\overline{x}+\gamma(\hat{y}-P_{V^{\bot}}(B+C)\overline{x})\in
	\zer (\mathcal{A}_\gamma + \mathcal{B}_\gamma
	+\mathcal{C}_\gamma)$ and \eqref{eq:solABC} implies that
	$\overline{x} \in Z$ and that $\overline{y}\in
	V^{\bot}\cap(A\overline{x}+P_{V}(B+C)\overline{x})$.
	Moreover, from the weakly continuity of $P_V$ and
	$P_{V^{\bot}}$,
	we obtain  $x_n =
	P_V z_n \weak P_{V} \overline{z} = \overline{x}$ and
	$y_{n}=P_{V^{\bot}}z_{n}/\gamma\rightharpoonup
	P_{V^{\bot}}\overline{z}/\gamma=\overline{y}$, which
	completes the proof.
\end{proof}

The sequence $(\lambda_n)_{n\in\N}$ in Algorithm~\ref{algo:1} can
be manipulated in order to accelerate the convergence.
However, as in
\cite{briceno2015Optim,Briceno2015JOTA,Spingarn1985MP}, the
inclusion in \eqref{e:algomain} is not always easy to solve.
The following result provides a particular case of our
method, in which this inclusion can be explicitly computed in terms of
the resolvent of $A$.

\begin{cor} \label{teo:main2}
	In the context of Problem~\ref{prob:main}, let $(x_0,y_0) \in
	V\times
	V^{\bot}$, let $\chi\in\RPP$ be the constant defined in
	\eqref{e:defchi}, let $\gamma\in \left]0,\chi\right[$, and let
	$(x_n)_{n
		\in \N}$ and $(y_n)_{n \in \N}$ be the sequences generated by
	the
	recurrence
	\begin{equation}
		\label{e:algomain2}
		\left\lfloor
		\begin{aligned}
			&\text{for }n=1,2,\ldots\\
			&p_{n}=J_{{\gamma}A}\big(x_{n}+\gamma
			y_n-{\gamma}P_{V}(B+C)x_{n}\big)\\
			&r_n=P_Vp_n\\
			&x_{n+1}=r_n+\gamma P_V(Bx_n-Br_{n})\\
			&y_{n+1}=y_{n}-\frac{p_n-r_n}{\gamma}.
		\end{aligned}
		\right.
	\end{equation}
	Then,  there exist
	$\overline{x}\in Z$ and $\overline{y}\in
	V^{\bot}\cap(A\overline{x}+P_{V}(B+C)\overline{x})$
	such that $x_{n}\weak \overline{x}$ and $y_{n}\weak
	\overline{y}$.
\end{cor}

\begin{proof}
	Note that \eqref{e:algomain2} implies that $(x_{n})_{n\in\N}$ and
	$(y_n)_{n \in \N}$ are sequences in $V$ and
	$V^{\bot}$, respectively.
	Fix  $n\in\N$ and set $q_n = (x_{n}+\gamma
	y_n-{\gamma}P_{V}(B+C)x_{n} - p_n )/\gamma$.
	Hence, we obtain from \eqref{e:algomain2} that
	$p_n+\gamma	
	q_n=x_{n}+\gamma y_n-{\gamma}P_{V}(B+C)x_{n}$, that
	$q_n \in A p_n$, that $x_{n+1}=P_{V}p_{n}+\gamma
	P_{V}(Bx_{n}-BP_{V}p_{n})$, and that
	$y_{n+1}=y_n-(p_n-P_Vp_n)/\gamma=y_n-P_{V^{\bot}}p_n/\gamma=
	P_{V^{\bot}}q_n$.
	Therefore, \eqref{e:algomain2} is a particular case of
	Algorithm~\ref{algo:1} when $\lambda_{n}\equiv 1 \in\, ] 0 ,
	\chi/\gamma [$ and the result
	hence follows from Theorem~\ref{teo:main1}.
\end{proof}	

\begin{rem}
	\begin{enumerate}
		\item
		Note that, in the case when $C=0$, \eqref{e:algomain2} reduces
		to
		the method proposed in \cite{Briceno2015JOTA}. Observe that
		in this case we can take $\beta\to\pinf$ which yields
		$\chi\to 1/L$.
		\item
		Note that, in the case when $B=0$, \eqref{e:algomain2} reduces
		to
		the method proposed in \cite{briceno2015Optim}. In this case, we
		can
		take $L\to 0$, which yields $\chi\to 2\beta$.
		\item In the case when $V=\H$, \eqref{e:algomain2} reduces to
		the
		algorithm proposed in \cite{Arias2017A} (see also
		Proposition~\ref{prop:pre}).
	\end{enumerate}
\end{rem}

\section{Applications}\label{sec:application}
In this section we tackle the following composite
primal-dual monotone inclusion.
\begin{pro}
	\label{prob:PD}
	Let $\sH$ be a real Hilbert space, let $\sV$ be a closed vector
	subspace of $\sH$,
	let $\sA\colon\sH\to 2^{\sH}$ be maximally monotone, let
	$\sM\colon\sH\to\sH$
	be monotone and $\mu$-Lipschitzian, for some $\mu\in\RPP$,
	let $\sC\colon \sH
	\to\sH$ be $\zeta$-cocoercive, for some $\zeta \in\RPP$, and let
	$m$ be a strictly positive integer. For every $i
	\in\{1, \ldots, m\}$, let $\sG_{i}$ be a real Hilbert space, let
	$\sB_{i}\colon
	\sG_{i} \to2^{\sG_{i}}$ be maximally monotone, let $\sN_{i}\colon
	\sG_{i} \to 2^{\sG_{i}}$ be monotone and
	such that $\sN_{i}^{-1}$ is $\nu_{i}$-Lipschitzian, for some
	$\nu_{i}
	\in\RPP$, let $\sD_i$ be maximally monotone and
	$\delta_i$-strongly
	monotone, for some $\delta_i\in\RPP$,
	and let $\sL_{i}\colon \sH\to \sG_{i}$ be a nonzero
	bounded
	linear operator.
	The problem is to
	\begin{multline}
		\label{e:KTinc}
		\text{find}\quad \overline{\sx} \in \sH,\overline{\su}_{1} \in
		\sG_{1}, \ldots,
		\overline{\su}_{m} \in
		\sG_{m}\:\:\quad\text{such that}\\\:\:
		\begin{cases}
			0 &\in \sA \overline{\sx}+\sM\overline{\sx}+\sC \overline{\sx}+
			\sum_{i=1}^{m}
			\sL_{i}^{*}\overline{\su}_i+N_{\sV}\overline{\sx}\\
			0&\in\big(\sB_{1}^{-1} +\sN_1^{-1}
			+\sD_{1}^{-1}\big)\overline{\su}_1-\sL_{1}
			\overline{\sx}\\
			&\:\vdots\\
			0&\in\big(\sB_{m}^{-1} +\sN_m^{-1}
			+\sD_{m}^{-1}\big)\overline{\su}_m-\sL_{m}
			\overline{\sx},
		\end{cases}
	\end{multline}
	under the assumption that the solution set $\boldsymbol{Z}$
	to \eqref{e:KTinc} is nonempty.
\end{pro}
Note that, if
$(\overline{\sx},\overline{\su}_1,\ldots,\overline{\su}_m)\in
\boldsymbol{Z}$ then $\overline{\sx}$ solves the primal inclusion
\begin{equation}
	\label{e:priminc}
	\text{find}\quad \overline{\sx} \in \sH\quad\text{such that}\quad
	0 \in \sA \overline{\sx}+\sM\overline{\sx}+\sC \overline{\sx}+
	\sum_{i=1}^{m} \sL_{i}^{*}\left(\left(\sB_{i} \infconv\sN_i
	\infconv\sD_{i}\right)\sL_{i}
	\overline{\sx}\right)+N_{\sV}\overline{\sx}
\end{equation}
and $(\overline{\su}_1,\ldots,\overline{\su}_m)$ solves the dual
inclusion
\begin{multline}
	\label{e:dualinc}
	\text{find}\:\: \overline{\su}_{1} \in \sG_{1}, \ldots,
	\overline{\su}_{m} \in
	\sG_{m}\:\:\text{such that}\\
	(\exists \sx \in
	\sH)\quad
	\begin{cases}
		-\sum_{i=1}^{m} \sL_{i}^{*}
		\overline{\su}_{i} \in \sA \sx+\sM \sx+\sC\sx+N_{\sV}\sx \\
		(\forall i \in\{1, \ldots,
		m\})\quad  \overline{\su}_{i}
		\in\big(\sB_{i} \infconv\sN_i\infconv\sD_{i}\big)\sL_{i} \sx.
	\end{cases}
\end{multline}

In the case when $\sV=\sH$, $\sC=0$, and, for every
$i\in\{1,\ldots,m\}$,
$\sD_i^{-1}=0$, this problem can be solved by algorithms in
\cite{Comb13,CombPes12} by using
Tseng's splitting \cite{Tseng2000SIAM} in a suitable product space.
In the case when $\sV=\sH$, $\sM=0$, and, for every
$i\in\{1,\ldots,m\}$,
$\sN_i^{-1}=0$, this problem can be solved by algorithms in
\cite{Combettes2014Optimization,Vu13} by using
forward-backward splitting in a suitable product space.
Since $\mathsf{M}+\mathsf{C}$ and
$(\mathsf{N}_i^{-1}+\mathsf{D}_i^{-1})_{1\le i\le m}$ are monotone
and Lipschitzian and $N_{\mathsf{V}}$ is maximally monotone,
Problem~\ref{prob:PD} can be solved by the algorithms in
\cite{Comb13,CombPes12}. However, these methods do not exploit
the cocoercivity or the vector subspace structure of
Problem~\ref{prob:PD}.
Other algorithms as those in
\cite{Combettes2018MP,EcksteinJohnstone,Bui2021} provide
alternatives for solving
Problem~\ref{prob:PD}, but any of them exploit its vector subspace
and cocoercive structure. In the case when $\sM=0$, and, for every
$i\in\{1,\ldots,m\}$,
$\sN_i^{-1}=0$, the algorithm in \cite{bdls} exploits the vector
subspace structure of Problem~\ref{prob:PD} by using the partial
inverse of $A$ with respect to $V$.
The following result provides a fully split
algorithm to solve Problem~\ref{prob:PD} in its full generality.
It is obtained by using \eqref{e:algomain2} in a
suitable
product space, which exploits the vector subspace structure and
which activates each cocoercive operator only once by iteration.

\begin{prop}
	\label{p:PD}
	Consider the framework of Problem~\ref{prob:PD} and set
	\begin{equation}
		L=\max\{\mu,\nu_1,\ldots,\nu_m\}+\sqrt{\sum_{i=1}^m\|\mathsf{L}_i\|^2}
		\quad\text{and}\quad\beta=\min\{\zeta,\delta_1,\ldots,\delta_m\}.
	\end{equation}
	Let $\sx_0\in\sV$, let $\sy_0\in\sV^{\top}$, for every
	$i\in\{1,\ldots,m\}$, let
	$\su_{i,0}\in\sG_i$,
	set $\gamma\in\left]0,\chi\right[$,
	where $\chi$ is defined in \eqref{e:defchi},
	and consider the routine
	\begin{equation}
		\label{e:algoPD}
		\left\lfloor
		\begin{aligned}
			&\text{for }n=1,2,\ldots\\
			&\mathsf{p}_{n}=J_{{\gamma}\sA}\bigg(\sx_{n}+\gamma
			\sy_n-{\gamma}P_{\sV}\bigg((\sM+\sC)\sx_{n}+
			\sum_{i=1}^m\sL_i^*\su_{i,n}\bigg)\bigg)\\
			&\sq_n=P_{\sV}\mathsf{p}_n\\
			&\left\lfloor
			\begin{aligned}
				&\text{for }i=1,\ldots,m\\
				&\sr_{i,n}=J_{\gamma\sB_{i}^{-1}}\Big(\su_{i,n}-\gamma
				\big((\sN_{i}^{-1}+\sD_{i}^{-1})\su_{i,n}-\sL_i\sx_n\big)\Big)\\
				&\su_{i,n+1}=\mathsf{r}_{i,n}-\gamma\big(\sN_i^{-1}\mathsf{r}_{i,n}
				-\sN_i^{-1}\mathsf{u}_{i,n}-\sL_i(\sq_n-\sx_n)\big)
			\end{aligned}
			\right.\\
			&\sx_{n+1}=\sq_n-\gamma P_{\sV}\bigg(\sM \sq_n-\sM
			\sx_{n}+\sum_{i=1}^m\sL_i^*(\mathsf{r}_{i,n}-\su_{i,n})\bigg)\\
			&\sy_{n+1}=\sy_{n}-\frac{\mathsf{p}_n-\sq_n}{\gamma}.
		\end{aligned}
		\right.
	\end{equation}
	Then, $(\sx_n)_{n\in\N}$ is a sequence in $\sV$ and there exists
	$(\overline{\sx},\overline{\su}_1,\ldots,\overline{\su}_m)\in\boldsymbol{Z}$
	such that $\sx_n\weakly\overline{\sx}$ and,
	for every $i\in\{1,\ldots,m\}$, $\su_{i,n}\weakly\overline{\su}_i$.
\end{prop}
\begin{proof}
	Set $\H=\sH\oplus\sG_1\oplus\cdots\oplus\sG_m$ and define
	\begin{equation}
		\begin{cases}
			A\colon\H\to 2^{\H}\colon (\sx,\su_1,\ldots,\su_m)\mapsto
			\sA\sx\times
			\sB_1^{-1}\su_1\times\cdots\times \sB_m^{-1}\su_m\\
			B\colon\H\to \H\colon (\sx,\su_1,\ldots,\su_m)\mapsto
			\big(\sM\sx+\sum_{i=1}^m\sL_i^*\su_i,\sN_1^{-1}\su_1-\sL_1\sx,\ldots,
			\sN_m^{-1}\su_m-\sL_m\sx\big)\\
			C\colon\H\to \H\colon (\sx,\su_1,\ldots,\su_m)\mapsto
			(\sC\sx,\sD_1^{-1}\su_1,\ldots,\sD_m^{-1}\su_m)\\
			V=\menge{(\sx,\su_1,\ldots,\su_m)\in\H}{\sx\in\sV}.
		\end{cases}
	\end{equation}
	Then, $A$ is maximally monotone and
	$B$ is monotone and $L$-Lipschitzian
	\cite[eq.(3.11)]{CombPes12},
	$C$ is $\beta$-cocoercive \cite[eq.(3.12)]{Vu13},
	and $V$ is a closed vector subspace of $\H$. Therefore,
	Problem~\ref{prob:PD} is a particular instance of
	Problem~\ref{prob:main}. Moreover, we have from
	\cite[Proposition~23.18]{bauschkebook2017} that
	\begin{equation}
		\begin{cases}
			(\forall \gamma>0)\quad J_{\gamma A}\colon
			(\sx,\su_1,\ldots,\su_m)\mapsto
			(J_{\gamma\sA}\sx,J_{\gamma\sB_1^{-1}}\su_1,\ldots,
			J_{\gamma\sB_m^{-1}}\su_m)\\
			P_{V}\colon
			(\sx,\su_1,\ldots,\su_m)\mapsto (P_{\sV}\sx,\su_1,\ldots,\su_m).
		\end{cases}
	\end{equation}
	Altogether, by defining
	\begin{equation}
		(\forall n\in\N)\quad
		\begin{cases}
			x_n=(\sx_n,\su_{1,n},\ldots,\su_{m,n})\\
			y_n=(\sy_n,0,\ldots,0)\\
			p_n=(\mathsf{p}_n,\sr_{1,n},\ldots,\sr_{m,n})\\
			q_n=(\sq_n,\mathsf{s}_{1,n},\ldots,\mathsf{s}_{m,n}),
		\end{cases}
	\end{equation}
	\eqref{e:algoPD} is a particular case of \eqref{e:algomain}
	and the convergence follows from Corollary~\ref{teo:main2}.
\end{proof}
\begin{rem}
	In the particular case when $\sV=\sH$ and
	$\sC=\sD_1^{-1}=\cdots=\sD_m^{-1}=0$, Proposition~\ref{p:PD}
	recovers the main result in \cite[Theorem~3.1]{CombPes12} in
	the
	error-free case. By including non-standard metrics in the space
	$\H$
	as in \cite{Arias2017A},
	we can also recover \cite{bdls} when
	$\sM=\sN_1^{-1}=\cdots=\sN_m^{-1}=0$ and \cite{Vu13} if we
	additionally assume that $\sV=\sH$, but we preferred to
	avoid this generalization for simplicity.
\end{rem}

We now provide two important examples of Problem~\ref{prob:PD}
and Proposition~\ref{p:PD} in the context of convex optimization.

\begin{ejem}
	Suppose that $\sA=\partial\mathsf{f}$,
	$\sM=\sN_1^{-1}=\cdots=\sN_m^{-1}=0$,
	$\sC=\nabla\mathsf{h}$, for
	every $i\in \{1,\ldots,m\}$,
	$\sD_i=\partial\ell_i$ and $\sB_i=\partial\sg_i$, where
	$\mathsf{f}\in\Gamma_0(\sH)$, $\mathsf{h}\colon\sH\to\R$ is
	convex differentiable with $\zeta^{-1}$-Lipschitzian gradient,
	for every $i\in \{1,\ldots,m\}$, $\ell_i\in\Gamma_0(\sG_i)$ is
	$\nu_i$-strongly convex and $\sg_i\in\Gamma_0(\sG_i)$.
	Then under the qualification condition
	\cite[Proposition~4.3(i)]{CombPes12}
	\begin{equation}
		\label{e:condqual}
		(0,\ldots,0)\in\sri\Big(\times_{i=1}^m\big(\mathsf{L}_i(\sV\cap\dom
		\mathsf{f})-(\dom\mathsf{g}_i+\dom{\ell}_i)\big)\Big),
	\end{equation}
	Problem~\ref{prob:PD} is equivalent to
	\begin{equation}
		\label{e:probgen}
		\min_{\sx\in\sV}\bigg(\mathsf{f}(\sx)+\mathsf{h}(\sx)
		+\sum_{i=1}^m(\sg_i\infconv\ell_i)(\sL_i\sx)\bigg),
	\end{equation}
	which, in view of Proposition~\ref{p:PD}, can be solved by the
	algorithm
	\begin{equation}
		\label{e:algoPDop}
		\left\lfloor
		\begin{aligned}
			&\text{for }n=1,2,\ldots\\
			&\mathsf{p}_{n}=\prox_{{\gamma}\mathsf{f}}\bigg(\sx_{n}
			+\gamma
			\sy_n-{\gamma}P_{\sV}\bigg(\nabla\mathsf{h}(\sx_{n})+
			\sum_{i=1}^m\sL_i^*\su_{i,n}\bigg)\bigg)\\
			&\sq_n=P_{\sV}\mathsf{p}_n\\
			&\left\lfloor
			\begin{aligned}
				&\text{for }i=1,\ldots,m\\
				&\sr_{i,n}=\prox_{\gamma\sg_{i}^*}\big(\su_{i,n}-\gamma
				(\nabla\mathsf{\ell}_i^*(\su_{i,n})-\sL_i\sx_n)\big)\\
				&\su_{i,n+1}=\mathsf{r}_{i,n}+\gamma \sL_i(\sq_n-\sx_n)
			\end{aligned}
			\right.\\
			&\sx_{n+1}=\sq_n-\gamma
			P_{\sV}\bigg(\sum_{i=1}^m\sL_i^*(\mathsf{r}_{i,n}-\su_{i,n})\bigg)\\
			&\sy_{n+1}=\sy_{n}-\frac{\mathsf{p}_n-\sq_n}{\gamma},
		\end{aligned}
		\right.
	\end{equation}
	where $\sx_0 \in V$, $\sy_0 \in V^\bot$, for
	every
	$i\in\{1,\ldots,m\}$, $\su_{i,0} \in \sG_i$, $L = \sqrt{\sum_{i=1}^m\|
		\mathsf{L}_i\|^2}$, $\beta=\min\{\zeta, \delta_i,\ldots,\delta_m\}$,
	$\chi$ is defined in \eqref{e:defchi},
	and $\gamma\in \left]0,\chi\right[$.
	Observe that the algorithm \eqref{e:algoPDop} exploits the
	cocoercivity of $\nabla \mathsf{h}$ and $(\nabla\ell_i^*)_{1\le i\le
		m}$
	by implementing them only once by iteration a difference of
	\cite[Theorem~4.2]{CombPes12}, which needs to implement
	them
	twice by iteration.
\end{ejem}

\begin{ejem}
	\label{ex:teonum}
	Consider the convex minimization problem
	\begin{equation}
		\label{e:probvu}
		\min_{\mathpzc{x}\in\mathpzc{H}}
		\big(\mathpzc{f}(\mathpzc{x})+\mathpzc{g}(\mathpzc{L}\mathpzc{x})
		+\mathpzc{h}(\mathpzc{A}\mathpzc{x})\big),
	\end{equation}
	where $\mathpzc{H}$, $\mathpzc{G}$, and $\mathpzc{K}$ are
	real Hilbert spaces,
	$\mathpzc{f}\in\Gamma_0(\mathpzc{H})$,
	$\mathpzc{g}\in\Gamma_0(\mathpzc{G})$,
	$\mathpzc{L}\colon\mathpzc{H}\to\mathpzc{G}$,
	$\mathpzc{A}\colon\mathpzc{H}\to\mathpzc{K}$,
	$\mathpzc{h}\colon\mathpzc{K}\to\RR$ is convex, differentiable
	with $\beta^{-1}$-Lipschitzian
	gradient, and suppose that
	\begin{equation}
		\label{e:CQ1}
		0\in\sri(\mathpzc{L}\,\dom\mathpzc{f}-\dom\mathpzc{g}).
	\end{equation}
	Note that
	$\mathpzc{h}\circ \mathpzc{A}$ is convex, differentiable, and
	$\nabla (\mathpzc{h}\circ \mathpzc{A})=\mathpzc{A}^*\circ
	\nabla \mathpzc{h}\circ \mathpzc{A}$ is
	$\beta^{-1}\|\mathpzc{A}\|^2-$Lipschitzian. Then,
	\eqref{e:probvu} can be solved
	by the primal-dual algorithm proposed in \cite{Condat13,Vu13},
	whose convergence is guaranteed under the assumption
	\begin{equation}
		\sigma\|\mathpzc{L}\|^2\le\frac{1}{\tau}-\frac{\|\mathpzc{A}\|^2}{2\beta},
	\end{equation}
	where $\tau>0$ and $\sigma>0$ are primal and dual step-sizes,
	respectively. Observe that, when $\|\mathpzc{A}\|$ is large, this
	method
	is forced to choose small primal and dual step-sizes in order to
	ensure
	convergence. To overcome this inconvenient, we propose the
	following formulation
	\begin{equation}
		\min_{\sx\in\sV}\big(\mathsf{f}(\sx)+\mathsf{h}(\sx)+\sg(\sL\sx)\big),
	\end{equation}
	where
	\begin{equation}
		\label{e:defserif}
		\begin{cases}
			\sH=\mathpzc{H}\oplus\mathpzc{K}\\
			\sG=\mathpzc{G}\\
			\sT\colon \sx = (\mathpzc{x},\mathpzc{w})\mapsto
			\mathpzc{A}\mathpzc{x}-\mathpzc{w}\\
			\sV=\ker\sT\\
			\mathsf{f}\colon \sx=(\mathpzc{x},\mathpzc{w})\mapsto
			\mathpzc{f}(\mathpzc{x})\\
			\mathsf{g}=\mathpzc{g}\\
			\sL\colon \sx=(\mathpzc{x},\mathpzc{w})\mapsto
			\mathpzc{L}\mathpzc{x}\\
			\mathsf{h}\colon \sx=(\mathpzc{x},\mathpzc{w})\mapsto
			\mathpzc{h}(\mathpzc{w}).
		\end{cases}
	\end{equation}
	Since in this case \eqref{e:condqual} reduces to \eqref{e:CQ1},
	\eqref{e:probvu} is a particular instance of
	\eqref{e:probgen} when $m=1$ and $\ell_1=0$. Therefore, in
	view of \cite[Example 29.19]{bauschkebook2017},
	\eqref{e:probvu} can be solved by the routine in
	\eqref{e:algoPDop} which, on this setting, reduces to:
	
	\begin{equation}
		\label{e:algoPDopex1}
		\left\lfloor
		\begin{aligned}
			&\text{for }n=1,2,\ldots\\
			&\mathpzc{p}_{1,n}=\prox_{{\gamma}\mathpzc{f}}\bigg(\mathpzc{x}_{n}+\gamma
			\mathpzc{y}_{1,n}-{\gamma}\big(\mathpzc{L}^{*}\mathpzc{u}_n
			-
			\mathpzc{A}^*\mathpzc{B}\big(\mathpzc{A}
			\mathpzc{L}^{*} \mathpzc{u}_n -\nabla \mathpzc{h}
			(\mathpzc{w}_{n}) \big) \big)\bigg)\\
			&\mathpzc{p}_{2,n}=\mathpzc{w}_{n}+\gamma
			\mathpzc{y}_{2,n}-{\gamma}\big( \nabla \mathpzc{h}
			(\mathpzc{w}_{n})+
			\mathpzc{B}\big(\mathpzc{A}
			\mathpzc{L}^{*} \mathpzc{u}_n -\nabla \mathpzc{h}
			(\mathpzc{w}_{n})  \big)\big)\\
			&\mathpzc{q}_{1,n}=\mathpzc{p}_{1,n}-
			\mathpzc{A}^*\mathpzc{B} \big(\mathpzc{A}\mathpzc{p}_{1,n}
			-\mathpzc{p}_{2,n} \big) \\
			&\mathpzc{q}_{2,n}=\mathpzc{p}_{2,n}+\mathpzc{B}\big(\mathpzc{A}\mathpzc{p}_{1,n}
			-\mathpzc{p}_{2,n} \big) \\
			&\mathpzc{r}_n=\prox_{\gamma\mathpzc{g}^*}\big(\mathpzc{u}_{n}+\gamma
			\mathpzc{L}\mathpzc{x}_n\big)\\
			&\mathpzc{u}_{n+1}=\mathpzc{r}_{n}+\gamma
			\mathpzc{L}(\mathpzc{q}_{1,n}-\mathpzc{x}_n)\\
			&\mathpzc{x}_{n+1}=\mathpzc{q}_{1,n}-\gamma
			\left(\mathpzc{L}^{*}(\mathpzc{r}_n-\mathpzc{u}_n)-
			\mathpzc{A}^*\mathpzc{B} \mathpzc{A}
			\mathpzc{L}^{*} (\mathpzc{r}_n-\mathpzc{u}_n) \right)\\
			&\mathpzc{w}_{n+1}=\mathpzc{q}_{2,n}-\gamma
			\mathpzc{B} \mathpzc{A}
			\mathpzc{L}^{*} (\mathpzc{r}_n-\mathpzc{u}_n)\\
			&\mathpzc{y}_{1,n+1}=\mathpzc{y}_{1,n}-\frac{\mathpzc{p}_{1,n+1}
				-\mathpzc{q}_{1,n+1}}{\gamma}\\		
			&\mathpzc{y}_{2,n+1}=\mathpzc{y}_{2,n}-
			\frac{\mathpzc{p}_{2,n+1}-\mathpzc{q}_{2,n+1}}{\gamma},
		\end{aligned}
		\right.
	\end{equation}
	where $\mathpzc{B}=(\id+\mathpzc{A}\mathpzc{A}^*)^{-1}$
	can
	be computed only once before the loop,
	$(\mathpzc{x}_0,\mathpzc{w}_0)
	\in V$, $(\mathpzc{y}_{1,0},\mathpzc{y}_{2,0}) \in V^{\bot}$,
	$\mathpzc{u}_{0} \in \mathpzc{G}$, $L=\|\mathpzc{L}\|$,
	$\chi$ is defined in \eqref{e:defchi},
	and  $\gamma\in \left]0,\chi\right[$.
\end{ejem}
\section{Numerical experiments}\label{sec:numerical}
In this section we consider the following optimization problem
\begin{equation}
	\label{eq:numericproblem}
	\min_{ \mathpzc{y}^0 \leq \mathpzc{x}\leq \mathpzc{y}^1}
	\left(\frac{\alpha_1}{2}\|\mathpzc{A}\mathpzc{x}-\mathpzc{z}\|^2 +
	\alpha_2 \|\nabla \mathpzc{x}\|_1\right),
\end{equation}
where $\mathpzc{y}^0=(\eta_i^0)_{1\le i\le N}$,
$\mathpzc{y}^1=(\eta_i^1)_{1\le i\le N}$ are vectors in $\R^N$,
$\alpha_1$ and $\alpha_2$ are in $\RPP$,
$\mathpzc{A}\in
\R^{K\times N}$, $\mathpzc{z} \in \R^K$, and
$\nabla: \R^N \to \R^{N-1} : (\xi_i)_{1\leq i\leq N} \mapsto
(\xi_{i+1}-\xi_i)_{1\leq i\leq N-1}$ is the discrete gradient. This problem
appears when computing the fusion estimator in fused LASSO
problems \cite{Friedman2007,Ohishi2021,Tibshirani2005}.

Note that \eqref{eq:numericproblem} can be written equivalently as
\eqref{e:probvu}, where
\begin{equation}
	\begin{cases}
		\mathpzc{H}=\R^N \\
		\mathpzc{f}=\iota_{\mathpzc{C}}\\
		\mathpzc{C}=\times_{i=1}^N[\eta_i^0,\eta_i^1]\\
		\mathpzc{g}=\alpha_2 \|\cdot\|_1\\
		\mathpzc{h}=\alpha_1\|\cdot-\mathpzc{z}\|^2/2\\
		\mathpzc{L}=\nabla
	\end{cases}
\end{equation}
Since $\mathpzc{f}\in\Gamma_0(\R^N)$, $\mathpzc{g}\in\Gamma_0(\R^{N-1})$,
$\mathpzc{h}$ is convex, differentiable,
$\nabla \mathpzc{h}= \alpha_1(\id-\mathpzc{z})$ is
$\alpha_1-$Lipschitzian,
$\|\mathpzc{L}\|=2$, and \eqref{e:CQ1} is trivially satisfied,
\eqref{eq:numericproblem} is a particular instance of
Example~\ref{ex:teonum}. Hence, \eqref{eq:numericproblem}
can be solved by  the algorithm in
\cite{Condat13,Vu13} (called Condat-V\~u), by
\eqref{e:algoPDopex1} (called FPIHF), and by
\cite{Briceno2015JOTA} (called FPIF), which are compared in this
section.
In this context,
the Algorithm {Condat-V\~u} \cite{Condat13,Vu13} reduces to the
following routine.
\begin{algorithm}[H]
	\caption{Condat-V\~u \cite{Condat13,Vu13}}
	\label{alg:cv}
	\begin{algorithmic}[1]
		\STATE{Let $\mathpzc{x}_0\in\R^{N}$ and $\mathpzc{u}_0 \in
			\R^{N-1}$,
			let $(\sigma,\tau,\rho)
			\in \RPP^3$, and fix $\epsilon_0 > \varepsilon>0$. }
		\WHILE{ $\epsilon_n > \varepsilon $}	
		\STATE{$\mathpzc{p}_{n+1}=P_\mathpzc{C}\big(\mathpzc{x}_n
			-\tau(\alpha_1
			\mathpzc{A}^{\top}(\mathpzc{A}\mathpzc{x}_n-z)+\nabla^{\top}\mathpzc{u}_n)\big)$}
		\STATE{$\mathpzc{q}_{n+1}=\sigma (\id-\prox_{
				\alpha_2\|\cdot\|_1/\sigma})
			(\mathpzc{u}_n/\sigma+\nabla
			(2\mathpzc{p}_{n+1}-\mathpzc{x}_n))$}
		\STATE{$\mathpzc{x}_{n+1}=\mathpzc{x}_n+\rho(\mathpzc{p}_{n+1}-\mathpzc{x}_n)$}
		\STATE{$\mathpzc{u}_{n+1}=\mathpzc{u}_n+\rho(\mathpzc{q}_{n+1}-\mathpzc{u}_n)$}
		\STATE{$\epsilon_{n+1}=\mathcal{R}\big((\mathpzc{x}_{n+1},\mathpzc{u}_{n+1}),
			(\mathpzc{x}_n,\mathpzc{u}_{n})\big)$}
		\ENDWHILE
		\RETURN{$(\mathpzc{x}_{n+1},\mathpzc{u}_{n+1})$}
	\end{algorithmic}
\end{algorithm}
Observe that $P_{\mathpzc{C}}\colon(\xi_i)_{1\le i\le
	N}\mapsto(\max\{\min\{\xi_{i},\eta_i^1\},\eta_i^0\})_{1\leq i\leq
	N}$.
The convergence of Algorithm~\ref{alg:cv} is guaranteed if
\begin{equation}\label{eq:tsvu}
	\sigma\|\mathpzc{L}\|^2\le\frac{1}{\tau}-\frac{\alpha_1
		\|\mathpzc{A}\|^2}{2}\quad\text{and}\quad
	\rho\in\left]0,\delta\right[,
	\quad\text{where}\:\:\delta=2-\frac{\alpha_1 \|\mathpzc{A}\|^2}
	{2(\frac{1}{\tau}-\sigma\|\mathpzc{L}\|^2)}.
\end{equation}
Note that, the larger is $\alpha_1 \|\mathpzc{A}\|^2$, the smaller
should be
$\tau$ and $\sigma$ in order to achieve convergence.
On the other hand, by considering $\mathsf{T}$
defined in \eqref{e:defserif}, the method in \eqref{e:algoPDopex1} writes as follows.
\begin{algorithm}[H]
	\caption{Forward-partial inverse-half-forward splitting (FPIHF)}
	\label{alg:TV}
	\begin{algorithmic}[1]
		\STATE{Set $\mathpzc{B}=(\id  + \mathpzc{A}
			\mathpzc{A}^\top)^{-1}$, let $(\mathpzc{x}_0,\mathpzc{w}_0) \in
			(\ker\mathsf{T})^2 $, $(\mathpzc{y}_{1,0},\mathpzc{y}_{2,0}) \in
			(\ker\mathsf{T}^\bot)^2 $, $\mathpzc{u}_0 \in \R^K $, let
			$\displaystyle
			\gamma \in
			\RPP$, and fix $\epsilon_0> \varepsilon>0$.}
		\WHILE{ $\epsilon_n > \varepsilon $}	
		\STATE{$\mathpzc{p}_{1,n}=P_\mathpzc{C}\big(\mathpzc{x}_{n}
			+\gamma\mathpzc{y}_{1,n}-{\gamma}\big(\nabla^{\top}
			\mathpzc{u}_n-
			\mathpzc{A}^\top\mathpzc{B} \big(\mathpzc{A} \nabla^{\top}
			\mathpzc{u}_n
			-\alpha_1(\mathpzc{w}_{n}-\mathpzc{z}) \big) \big)\big)$}
		\STATE{$\mathpzc{p}_{2,n}=\mathpzc{w}_{n}+\gamma
			\mathpzc{y}_{2,n}-{\gamma}\big(
			\alpha_1(\mathpzc{w}_{n}-\mathpzc{z})+ \mathpzc{B}
			\big(\mathpzc{A} \nabla^{\top} \mathpzc{u}_n -
			\alpha_1(\mathpzc{w}_{n}-\mathpzc{z})  \big)\big)$}
		\STATE{$\mathpzc{q}_{1,n}=\mathpzc{p}_{1,n}-   \mathpzc{A}^\top
			\mathpzc{B} (\mathpzc{A}\mathpzc{p}_{1,n} -\mathpzc{p}_{2,n}
			\big)$}
		\STATE{$\mathpzc{q}_{2,n}=\mathpzc{p}_{2,n}+   \mathpzc{B}
			\big(\mathpzc{A}\mathpzc{p}_{1,n} -\mathpzc{p}_{2,n} \big)$}
		\STATE{$\mathpzc{r}_n= \gamma (\id-\prox_{
				\alpha_2\|\cdot\|_1/\gamma})(\mathpzc{u}_{n}/\gamma+\nabla
			\mathpzc{x}_n)$}
		\STATE{$\mathpzc{u}_{n+1}=\mathpzc{r}_{n}+\gamma
			\nabla(\mathpzc{q}_{1,n}-\mathpzc{x}_n)$}
		\STATE{$\mathpzc{x}_{n+1}=\mathpzc{q}_{1,n}-\gamma
			\left(\nabla^{\top}(\mathpzc{r}_n-\mathpzc{u}_n)-
			\mathpzc{A}^\top \mathpzc{B}\mathpzc{A}\nabla^{\top}
			(\mathpzc{r}_n-\mathpzc{u}_n)  \right)$}
		\STATE{$\mathpzc{w}_{n+1}=\mathpzc{q}_{2,n}-\gamma
			\mathpzc{B} \mathpzc{A} \nabla^{\top}
			(\mathpzc{r}_n-\mathpzc{u}_n) $}
		\STATE{$\mathpzc{y}_{1,n+1}=\mathpzc{y}_{1,n}
			-(\mathpzc{p}_{1,n+1}-\mathpzc{q}_{1,n+1})/\gamma$}		
		\STATE{$\mathpzc{y}_{2,n+1}=\mathpzc{y}_{2,n}-(\mathpzc{p}_{2,n+1}-\mathpzc{q}_{2,n+1})/\gamma$}
		\STATE{	
			$\epsilon_{n+1}=\mathcal{R}\big((\mathpzc{x}_{n+1},\mathpzc{w}_{n+1},\mathpzc{y}_{n+1}^1,\mathpzc{y}_{n+1}^2),
			(\mathpzc{x}_{n},\mathpzc{w}_{n},\mathpzc{y}_{n}^1,\mathpzc{y}_{n}^2)\big)$
		}
		\ENDWHILE
		\RETURN{$(\mathpzc{x}_{n+1},\mathpzc{w}_{n+1},\mathpzc{y}_{n+1}^1,\mathpzc{y}_{n+1}^2)$}
	\end{algorithmic}
\end{algorithm}
In view of Example~\ref{ex:teonum}, the algorithm  in
\eqref{e:algoPDopex1} reduces to Algorithm~\ref{alg:TV}, whose
convergence is guaranteed if the step-size $\gamma$ satisfies
\begin{equation} \label{e:gammaexp}
	0<\gamma < \chi = \frac{4}{\alpha_1+\sqrt{\alpha_1^2+64}}.
\end{equation}
Observe that the condition for the step-size $\gamma$ in
\eqref{e:gammaexp} does not
depend on $\|\mathpzc{A}\|$.

The FPIF algorithm proposed in  \cite{Briceno2015JOTA} for solving
\eqref{eq:numericproblem}
differs from Algorithm~\ref{alg:TV} in the fact that
the cocoercive gradient $\nabla\mathpzc{h}\colon \mathpzc{x}\mapsto
\mathpzc{x}-\mathpzc{z}$ is implemented twice by iteration.
Indeed, the algorithm consider the monotone Lipschitzian operator
$(\mathpzc{x},\mathpzc{u},\mathpzc{w})\mapsto
(\nabla^{\top}\mathpzc{u},
\nabla \mathpzc{x},\alpha_1(\mathpzc{w}-\mathpzc{z}))$, whose
Lipschitz
constant follows from
\begin{align*}
	&\pnorm{\big(\nabla^{\top}\mathpzc{u}_1,
		\nabla \mathpzc{x}_1,\alpha_1(\mathpzc{w}_1-\mathpzc{z})\big)
		-\big(\nabla^{\top}\mathpzc{u}_2,
		\nabla
		\mathpzc{x}_2,\alpha_1(\mathpzc{w}_2-\mathpzc{z})\big)}^2\\
	& \hspace*{4cm}=
	\|\nabla^{\top}(\mathpzc{u}_1-\mathpzc{u}_2)\|^2+
	\|\nabla (\mathpzc{x}_1-\mathpzc{x}_2)\|^2 +
	\alpha_1^2\|\mathpzc{w}_1-\mathpzc{w}_2\|^2
	\\
	&\hspace*{4cm}\leq \|\nabla^{\top}\|^2
	\|\mathpzc{u}_1-\mathpzc{u}_2\|^2+
	\|\nabla\|^2 \|\mathpzc{x}_1-\mathpzc{x}_2\|^2 +
	\alpha_1^2\|\mathpzc{w}_1-\mathpzc{w}_2\|^2\\
	&\hspace*{4cm}\leq \max\{\|\nabla\|^2,\alpha_1^2\}
	\pnorm{(\mathpzc{x}_1-\mathpzc{x}_2,\mathpzc{u}_1
		-\mathpzc{u}_2,\mathpzc{w}_1-\mathpzc{w}_2)}^2.
\end{align*}	
Therefore, the convergence of FPIF is guaranteed if
$\gamma\in\left]0,1/\max\{\|\nabla\|,\alpha_1\}\right[$, and, as in
Algorithm~\ref{alg:TV}, this condition does not depend on
$\|\mathpzc{A}\|$.
In order to compare Condat-V\~u, FPIHF, and FPIF, we set
$\alpha_1=5$ and
$\alpha_2=0.5$ and we consider
$\mathpzc{A}=\kappa \cdot \text{rand}(N,K)$,
$\mathpzc{y}^0=-1.5\cdot\text{rand}(N)$,
$\mathpzc{y}^1=1.5\cdot\text{rand}(N)$,
and $\mathpzc{z}=\text{randn}(N)$,
where $\kappa \in \{1/5,1/10,1/20,1/30\}$,
$N\in\{600,1200,2400\}$,
$K\in\{N/3,N/2,2N/3\}$,  and
$\text{rand}(\cdot,\cdot)$ and $\text{randn}(\cdot,\cdot)$ are
functions in MATLAB generating matrices/vectors with uniformly and
normal distributed
entries, respectively. For each value of $\kappa$, $N$, and $K$,
we generate
$20$ random realizations for $\mathpzc{A}$,
$\mathpzc{z}$, $\mathpzc{y}^0$, and $\mathpzc{y}^1$.
Note that the average value of
$\|\mathpzc{A}\|$ increases as $\kappa$ increase
(see Figure~\ref{fig:boxplotjuntos} for $K=N/2$), which affects
Algorithm~\ref{alg:cv} in view of \eqref{eq:tsvu}. We also set
$\rho=0.99\cdot\delta$, where $\delta$ is defined in \eqref{eq:tsvu}.
In this setting, from \eqref{e:gammaexp} we deduce that the
convergence of FPIHF is guaranteed for $\gamma < \chi \approx
0.2771$. On the other hand, since $\max\{\|\nabla\|,\alpha_1\} =
\alpha_1=5$, the  convergence of FPIF is guaranteed for $\gamma <
0.2$.
\begin{figure}[H]
	\includegraphics[scale=0.4]{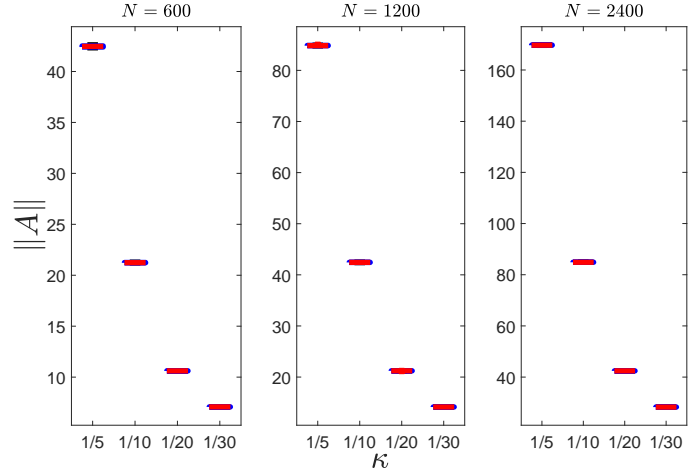}
	\caption{Box plot for the norm of the $20$ random realizations of
		$\mathpzc{A}$, $N \in \{600, 1200, 2400\}$,
		$K=N/2$.}\label{fig:boxplotjuntos}
\end{figure}

In Tables~\ref{t:IV}-\ref{t:I} we provide the
average time and number of iterations to achieve a tolerance
$\varepsilon=10^{-6}$ for each algorithm under study. In the case
when an algorithm exceeds
$50000$ iterations in all cases, we write ``$\boxtimes$'' in both
columns.
From these tables we can observe that when $\kappa$
increases (and therefore, $\|\mathpzc{A}\|$ increases), Condat-V\~u
reduces its performance and does not converge within
$50000$ iterations for big dimensions and large values of $\kappa$.
Moreover, the number of iterations of FPIHP is
considerably lower than its competitors but with expensive
computational time by iteration. This can be explained
by the fact that  FPIHP needs to compute
three projections onto the kernel of
$(\mathpzc{x},\mathpzc{w})\mapsto
\mathpzc{A}\mathpzc{x}-\mathpzc{w}$ at each iteration.
We can also perceive that, at exception of some cases, the partial
inverse-based algorithms increase their computational time to achieve
convergence when $K$ is larger. This can be explained by the fact
that the dimension of matrix $\mathpzc{B}$ is larger as $K$ is larger,
and it has to be implemented three times by iteration.

When $\kappa =1/30$, we observe from Table~\ref{t:IV},
that FPIHP and Condat-V\~u are competitive and both are more
efficient
than FPIF.
When $\kappa =1/20$,  we observe from Table~\ref{t:III} that FPIHP
outperforms Condat-V\~u and FPIF for large dimensions.
When $\kappa =1/10$, we observe from Table~\ref{t:II} that
FPIHP is the best algorithm at exception of the smallest dimensional
case in which it is competitive with Condat-V\~u. The latter does not
converge within 50000 for dimension $N=2400$.
When $\kappa =1/5$, FPIHP is the more efficient algorithm in all the
cases under study, as it is illustrated in Table~\ref{t:I}.
Moreover, Condat-V\~u converge before 50000 iterations only in the
lower
dimensional case  when $N=600$. We conclude that, for higher
values of $\|\mathpzc{A}\|$ and larger dimensions, is more convenient
to implement FPIHP.
\begin{table}
	{\caption{Comparison of Condat-V\~u, FPIF, and FPIHF for the
			case $\kappa=1/30$.\label{t:IV}}
		\begin{center}
			\begin{tabular}{cccccccc}
				\multicolumn{2}{c}{} & \multicolumn{2}{c}{$
					K=N/3$}&\multicolumn{2}{c}{$K=N/2$}
				&\multicolumn{2}{c}{ $K=2N/3$}\\  \cmidrule[0.5pt](lr){3-4}
				\cmidrule[0.5pt](lr){5-6}\cmidrule[0.5pt](lr){7-8}
				$N$ & Algorithm &  Av. time (s)& Av. iter & Av. time (s)&
				Av. iter &
				Av. time (s)& Av. iter  \\ \midrule
				\multirow{3}{*}{$600$}	&Condat-V\~u  &
				0.89  & 11059 & 0.80 & 10047 & 0.76 & 9666   \\
				&FPIF &  3.46 & 17454 & 3.91 & 14353 & 7.20 & 17430\\
				& FPIHF  &0.99 & 4851 & 1.24 & 4442 & 1.73 & 3996 \\ \midrule  	
				\multirow{3}{*}{$1200$}	&Condat-V\~u &
				11.32 & 17321 & 10.55 & 16129 &  10.54 & 16082   \\				
				&FPIF & 25.52 & 19930 & 32.37 & 13788 & 51.54 & 16443\\
				& FPIHF  & 7.07 & 5425 & 13.76 & 5838 & 23.83 & 7570\\
				\midrule
				\multirow{3}{*}{$2400$}	&Condat-V\~u  &
				74.17 & 34059 & 70.14 & 32216 & 69.48 & 31963   \\
				&FPIF & 95.55 & 17747 & 138.67 & 16074 & 190.06 & 17216\\
				& FPIHF  & 43.08 & 7961 & 64.68 & 7464 & 70.64 & 6369\\
				\midrule
			\end{tabular}
	\end{center}}
\end{table}

\begin{table}
	
	{\caption{Comparison of Condat-V\~u, FPIF, and FPIHF for the
			case  $\kappa=1/20$.\label{t:III}}
		\begin{center}
			\begin{tabular}{cccccccc}				\multicolumn{2}{c}{} &
				\multicolumn{2}{c}{$
					K=N/3$}&\multicolumn{2}{c}{$K=N/2$}
				&\multicolumn{2}{c}{ $K=2N/3$}\\   \cmidrule[0.5pt](lr){3-4}
				\cmidrule[0.5pt](lr){5-6}\cmidrule[0.5pt](lr){7-8}
				$N$ & Algorithm &  Av. time (s)& Av. iter & Av. time (s)&
				Av. iter &
				Av. time (s)& Av. iter  \\ \midrule
				\multirow{3}{*}{$600$}&Condat-V\~u  &
				0.86 & 10752 & 0.81 & 10263 & 0.87 & 10992   \\
				&FPIF  & 2.67 & 13381 & 3.91 & 14204 & 5.88 & 14258 \\
				&FPIHF  & 0.97 & 4725 & 0.82 & 2900 & 1.63 & 3747 	 \\ \midrule
				\multirow{3}{*}{$1200$}&Condat-V\~u  & 13.91
				& 21209 & 13.35 & 20359 & 12.51 & 19118  \\
				&FPIF  & 23.30 & 18142 & 45.16  & 19222 & 52.60 & 16773 \\
				&FPIHF  & 9.07 & 6943 & 20.53 & 8689 & 10.91 & 3458	 \\
				\midrule	\multirow{3}{*}{$2400$}&Condat-V\~u  &
				103.92& 47673 & 98.92 & 45543 & 91.33 & 41996   \\
				&FPIF  & 89.77 & 16659 & 132.60 & 15374 & 145.58 & 13181 \\
				&FPIHF  & 32.27 & 5957 & 45.35 & 5234 & 83.48 & 7539
				\\
				\midrule
			\end{tabular}
	\end{center}}
\end{table}

\begin{table}	
	{\begin{center}
			\caption{Comparison of Condat-V\~u, FPIF, and FPIHF for the case
				$\kappa=1/10$.\label{t:II}}
			\begin{tabular}{cccccccc}
				\multicolumn{2}{c}{} & \multicolumn{2}{c}{$
					K=N/3$}&\multicolumn{2}{c}{$K=N/2$}
				&\multicolumn{2}{c}{ $K=2N/3$}\\     \cmidrule[0.5pt](lr){3-4}
				\cmidrule[0.5pt](lr){5-6}\cmidrule[0.5pt](lr){7-8}
				$N$ & Algorithm &  Av. time (s)& Av. iter & Av. time (s)&
				Av. iter &
				Av. time (s)& Av. iter  \\ \midrule
				\multirow{3}{*}{$600$} &	Condat-V\~u
				&1.43  & 18233 & 1.30 & 16747 & 1.25 & 15577 \\
				& FPIF  & 3.56 & 18040 & 3.01 & 11057 & 5.17 & 12389\\
				& FPIHF & 1.11  & 5414 & 1.30 & 4696 & 1.49 & 3436	 \\
				\midrule
				\multirow{3}{*}{$1200$} &	Condat-V\~u  &
				30.19 & 46078 & 26.98& 41243 & 24.05 & 36849 \\
				& FPIF  & 25.61 & 19916 & 30.70 & 13095 & 40.57 & 12960\\
				& FPIHF & 6.96 & 5343 & 10.16 & 4294 & 17.79 & 5657	
				\\\midrule
				\multirow{3}{*}{$2400$} &	Condat-V\~u  &
				$\boxtimes$ & $\boxtimes$ & $\boxtimes$ & $\boxtimes$ &
				$\boxtimes$& $\boxtimes$ \\
				& FPIF  & 98.90 & 18363 & 129.27 & 14975 & 172.05 & 15609\\
				& FPIHF & 28.90 & 5349 & 46.74 & 5391 & 60.61 & 5484
				\\\midrule
			\end{tabular}
	\end{center}}
	
\end{table}

\begin{table}
	
	{\caption{Comparison of Condat-V\~u, FPIF, and FPIHF for the
			case  $\kappa=1/5$.\label{t:I}}
		\begin{center}
			\begin{tabular}{cccccccc}
				\multicolumn{2}{c}{} & \multicolumn{2}{c}{$
					K=N/3$}&\multicolumn{2}{c}{$K=N/2$}
				&\multicolumn{2}{c}{ $K=2N/3$}\\     \cmidrule[0.5pt](lr){3-4}
				\cmidrule[0.5pt](lr){5-6}\cmidrule[0.5pt](lr){7-8}
				$N$ & Algorithm &  Av. time (s)& Av. iter & Av. time (s)&
				Av. iter &
				Av. time (s)& Av. iter  \\ \midrule
				\multirow{3}{*}{$600$}
				& Condat-V\~u & 3.76 & 48078 & 3.27 & 40998  & 2.58 &
				33226  \\
				&FPIF  & 2.68 & 13527 & 3.31 & 11945 & 4.14 & 9840 \\
				&FPIHF  & 0.50 & 2428 & 0.64 & 2263 & 0.79 & 1780	 \\
				\midrule
				\multirow{3}{*}{$1200$}&
				Condat-V\~u & $\boxtimes$ & $\boxtimes$ & $\boxtimes$
				& $\boxtimes$ & $\boxtimes$ & $\boxtimes$  \\
				&FPIF & 21.26 & 16535 & 27.29 & 11627 & 35.55 &
				11399\\
				&FPIHF & 7.23 & 5529  & 5.72 & 2424 & 10.25 &	3257 \\
				\midrule
				\multirow{3}{*}{$2400$}& Condat-V\~u & $\boxtimes$
				& $\boxtimes$ &  $\boxtimes$ & $\boxtimes$ &
				$\boxtimes$ & $\boxtimes$  \\
				&FPIF  & 88.51 & 16392 & 124.71 & 14444 & 139.69 &
				12653\\
				&FPIHF  & 23.95 & 4414 & 35.51 & 4102 & 41.38 &
				3773	 \\ \midrule
			\end{tabular}
	\end{center}}
\end{table}

\section*{Acknowledgment}
The first author thanks the support of  Centro de Modelamiento
Matemático (CMM), ACE210010 and FB210005, BASAL funds for
centers of excellence and ANID under grant FONDECYT 1190871
from ANID-Chile. The third author thanks the support
of ANID-Subdirección de Capital Humano/Doctorado
Nacional/2018-21181024 and by the Direcci\'on de Postgrado y Programas from
UTFSM through Programa de Incentivos a la Iniciaci\'on Cient\'ifica
(PIIC). The forth author thanks the support of the National Natural Science Foundations of China (12061045, 11661056).


\end{document}